\documentclass{amsart}
\usepackage{graphicx}
\newtheorem{theorem}{Theorem}[section]

\newtheorem{proposition}[theorem]{Proposition}

\newtheorem{claim}[theorem]{Claim}
\theoremstyle{definition}
\newtheorem{definition}[theorem]{Definition}

\theoremstyle{remark}
\newtheorem{remark}[theorem]{Remark}

\numberwithin{equation}{section}

\begin{document}

\title{Composite tunnel number one genus two handlebody-knots}

\author{Mario Eudave-Mu\~{n}oz}
\address{Instituto de Matem\'{a}ticas, UNAM, Circuito Exterior, Ciudad Universitaria, 04510 M\'{e}xico, D.F., Mexico, and CIMAT, Guanajuato, Mexico}
\curraddr{}
\email{mario@matem.unam.mx}
\thanks{The first author is partially supported by PAPIIT-UNAM grant IN109811}

\author{Makoto Ozawa}
\address{Department of Natural Sciences, Faculty of Arts and Sciences, Komazawa University, 1-23-1 Komazawa, Setagaya-ku, Tokyo, 154-8525, Japan}
\email{w3c@komazawa-u.ac.jp}
\thanks{The second author is partially supported by Grant-in-Aid for Scientific Research (C) (No. 23540105), The Ministry of Education, Culture, Sports, Science and Technology, Japan}


\date{}

\dedicatory{This paper is dedicated to the 70th birthday of Professor Fico Gonz\'{a}lez Acu\~{n}a.}


\begin{abstract}
We characterize composite tunnel number one genus two handlebody-knots.
\end{abstract}

\maketitle

\section{Introduction}
It is a fundamental theorem in Knot Theory that any non-splittable link can be uniquely decomposed into prime links.
This theorem was proved for knots by Schubert (\cite{S}) and for links by Hashizume (\cite{H}).

It was expected that two generator knots can not be decomposed into prime knots since two generator knots are most ``simple'' among all knots.
Indeed, Norwood showed that two generator knots are prime (\cite{N}).
In contrast to Norwood's algebraic proof, Scharlemann gave a geometric proof showing that tunnel number one knots are prime (\cite{Sch}).
In this direction, Scharlemann showed that tunnel number one knots are 2-string prime (\cite{Sch}), and Gordon--Reid showed that tunnel number one knots are $n$-string prime for all $n$ (\cite{GR}).
Although Li gave a counterexample to the Rank versus Genus Conjecture for closed orientable hyperbolic 3-manifolds (\cite{L}), it remains to be unknown whether there exists a knot in $S^3$ such that the rank of the fundamental group of the knot exterior is less than the Heegaard genus of it (e.g. \cite[Question 2]{L}).

On the other hand, there are two generator, or tunnel number one links which are not prime.
Jones showed that composite two-generator links have a Hopf link summand (\cite{J}).
Furthermore, Morimoto showed that tunnel number one links are composite if and only if they are connected sums of 2-bridge knots and the Hopf link (\cite{M}, cf. \cite{EU}).
In this direction, Gordon--Reid showed that $n$-string composite tunnel number one links have a Hopf tangle summand (\cite{GR}, cf. \cite{GOT}).

Suzuki generalized Schubert--Hashizume's result to spatial graphs by proving that every connected graph embedded in $S^3$ can be split along spheres meeting the graph in 1 or 2 points to obtain a unique collection of prime embedded graphs together with some trivial graphs (\cite{Su}).
Motohashi and Matveev--Turaev proved a prime decomposition theorem for $\theta_n$-curves by decomposing spheres which intersect each edge in one point (\cite{Mo1}, \cite{MT}).
Motohashi also proved a prime decomposition theorem for handcuff graphs by decomposing spheres which intersects the graph in exactly three points (\cite{Mo2}).

In this direction, we consider a genus two handlebody embedded in the 3-sphere whose exterior admits a genus three Heegaard splitting, and 
which has a decomposing sphere intersecting the handlebody in two meridian disks.

\section{Results}

We call an embedding (or the image of it) of a handlebody $V$ into $S^3$ a {\em handlebody-knot}.
We denote the exterior $S^3-\text{int}V$ of $V$ by $E(V)$.

The following definition on decomposing spheres for handlebody-knots was given by Ishii, Kishimoto and Ozawa in \cite{IKO}.

\begin{definition}
A 2-sphere $S$ embedded in $S^3$ is an {\em $n$-decomposing sphere} for a handlebody-knot $V$ if the following conditions hold.
\begin{enumerate}
\item $S\cap V$ consists of $n$ essential disks of $V$.
\item $S\cap E(V)$ consists of an incompressible and not boundary-parallel planar surface in $E(V)$.
\end{enumerate}
\end{definition}

We say that a handlebody-knot $V$ is {\em reducible} if $E(V)$ is boundary-reducible, i.e. $\partial V=\partial E(V)$ is compressible in $E(V)$, and $V$ is {\em irreducible} if it is not reducible (note that $E(V)$ is always irreducible in the sense that any sphere bounds a 3-ball).
It follows that if $V$ has a 1-decomposing sphere, then $V$ is reducible.
Conversely, Tsukui showed that if $V$ is reducible, then $V$ has a 1-decomposing sphere in the case that the genus of $V$ is two (\cite{T}, cf. \cite{KO}).

The decomposition by 1-decomposing spheres is unique for a genus two handlebody-knot (\cite{T1}) and for a trivial handlebody-knot, i.e. one standardly embedded in $S^3$. The uniqueness is not known for a genus $g\ge 3$ handlebody-knot.
Ishii, Kishimoto and the second author showed that a handlebody-knot whose exterior is boundary-irreducible has a unique maximal unnested set of knotted handle decomposing
spheres up to isotopies and annulus-moves (\cite{IKO}).
Moreover,  Koda and the second author showed the same uniqueness theorem for arbitrary handlebody-knots (\cite{KO}).

When the genus $g$ of $V$ is greater than one, the spine of a handlebody-knot $V$ can not be uniquely determined.
However, if we specify $g-1$ mutually non-parallel essential disks in $V$, then the spine without degree one vertices can be uniquely determined.
For the case that the genus of $V$ is equal to two, if $V$ has a 2-decomposing sphere $S$, then by the definition (1), $S\cap V$ gives two mutually parallel essential disks in $V$.
Therefore, the spine of $V$ is uniquely determined by a theta-curve or handcuff graph depending on whether the disks of $S\cap V$ are non-separating or separating in $V$.

Let $\{\gamma_1,\ldots,\gamma_t\}$ be a set of mutually disjoint arcs properly embedded in $E(V)$.
We call the set $\{\gamma_1,\ldots,\gamma_t\}$ an {\em unknotting tunnel system} for $V$ if $E(V)-\text{int}N(\gamma_1\cup\ldots\cup\gamma_t)$ is a handlebody.
The {\em tunnel number} of $V$ is the minimal number of arcs among all unknotting tunnel systems.
We say that a handlebody-knot is {\em trivial} if the tunnel number is zero.
By Waldhausen's theorem, any two genus $g$ trivial handlebody-knots are equivalent up to isotopy of $S^3$ (\cite{W}).
When the tunnel number is one, we abbreviate $\{\gamma\}$ to $\gamma$ and call $\gamma$ an {\em unknotting tunnel}.
From the point of view of the tunnel number, the first considerable class of handlebody-knots is the class of tunnel number one handlebody-knots.
For a spatial graph $\Gamma$, the exterior, unknotting tunnel (system), and tunnel number are defined by that for $N(\Gamma)$.

Let $\Gamma$ be a connected spatial graph in $S^3$, and $S$ a 2-sphere which intersects $\Gamma$ transversely, hence it is disjoint from the vertices of $\Gamma$.
Then $(S^3,\Gamma)$ is decomposed by $S$ into two tangles $(B_1,T_1)$ and $(B_2,T_2)$.
By a tangle we mean a pair $(B,T)$ consisting of a 3-ball $B$ and a properly embedded graph $T$ in $B$, that is, $\partial B \cap T$ consists of the degree one vertices of $T$.
We say that a tangle $(B,T)$ is {\em free} if $B-\text{int}N(T)$ is a handlebody.
A tangle $(B,T)$ is {\em essential} if $\partial B\cap E(T)$ is incompressible and not boundary-parallel in $E(T)$, where $E(T)=B-\text{int}N(T)$ is the exterior of $T$.
We say that a tangle decomposition $(S^3,\Gamma)=(B_1,T_1)\cup_S(B_2,T_2)$ is {\em essential} if both tangles $(B_1,T_1)$ and $(B_2,T_2)$ are essential.
It follows that if $S$ is an essential tangle decomposing sphere for $\Gamma$, then $S$ is also an $n$-decomposing sphere for the handlebody-knot $V=N(\Gamma)$, where $n=|S\cap \Gamma|$.

Let $\Gamma$ be a theta-curve $\theta$ or a handcuff graph $\phi$ embedded in $S^3$, and $e$ be an edge of $\theta$ or the cut edge of $\phi$.
Let $P$ be a 2-sphere which intersects $\Gamma$ in four points or three points of $e$ depending on whether $\Gamma$ is $\theta$ or $\phi$.
Then $P$ bounds a tangle $(B,T)$, where $T$ consists of a cycle $C$ attached with two or one edges, and an arc $\alpha$.
We say that such a tangle $(B,T)$ is a {\em Hopf tangle with two edges} or a {\em Hopf tangle with one edge} if $T-\alpha$ is contained in a properly embedded disk $D$ in $B$, and $C$ bounds a subdisk $D'$ of $D$ which intersects $\alpha$ in one point, and $\alpha$ is an unknotted arc in $B$.

\begin{theorem}\label{main}
Let $V$ be an irreducible tunnel number one genus two handlebody-knot in $S^3$.
If $V$ has a 2-decomposing sphere $S$, then there exists a spine $\Gamma$ of $V$ which satisfies one of the following conditions.
\begin{enumerate}
\item $\Gamma$ is a theta-curve graph which is decomposed by $S$ into a tunnel number zero theta-curve graph $\theta$ and a $(1,1)$-knot $K$.
\item $\Gamma$ is a handcuff graph which is decomposed by $S$ into a tunnel number zero handcuff graph $\phi$ and a $(1,1)$-knot $K$, where $S$ intersects $\phi$ in the cut edge of $\phi$.
\item $\Gamma$ is a theta-curve graph which is decomposed by $S$ into a theta-curve graph $\theta$ and a 2-bridge knot $K$, where the connected sum $\Gamma=\theta\#_S K$ is done at an edge $e$ of $\theta$.
The theta-curve $\theta$ is decomposed by a 2-sphere $P$ into an essential free 2-string tangle and a Hopf tangle with two edges, where $P$ intersects $\theta$ in the edge $e$.
\item $\Gamma$ is a handcuff graph which is decomposed by $S$ into a handcuff graph $\phi$ and a 2-bridge knot $K$, where the connected sum $\Gamma=\phi\#_S K$ is done at the cut edge $e$ of $\phi$.
The handcuff graph $\phi$ is decomposed by a 2-sphere $P$ into an essential free tangle and a Hopf tangle with one edge, where $P$ intersects $\phi$ in the edge $e$.
\end{enumerate}
\end{theorem}

See Figure 1 for a sketch of the four types of a spine $\Gamma$. Note that if a graph $\Gamma$ satisfies one of the conclusions of the theorem, then in fact it has tunnel number one. Note also that a genus two handlebody-knot could have different spines satisfying different conclusions of the theorem.

Explicit examples of genus two handlebody-knots admitting 2-decomposing spheres can be found in the table given by Ishii, Kishimoto, Moriuchi and Suzuki \cite{IKMS}. They have produced a table of genus two handelbody-knots which have diagrams with at most six crossings. All these handlebody-knots have tunnel number one, and it follows by inspection that  the handlebody knots $5_4$ and $6_{14}$ satisfy condition (1) of the theorem, while $6_{15}$ and $6_{16}$ satisfy condition (2) of the theorem.

\begin{remark}
In the proof of Theorem \ref{main}, we will show that there exists an unknotting tunnel $\gamma$ for $V$ which intersects $S$ in at most one point.
If $|\gamma\cap S|=0$, then we have a conclusion (1) or (2).
If $|\gamma\cap S|=1$, then we have a conclusion (3) or (4).
\end{remark}

\begin{figure}[htbp]
	\begin{center}
	\begin{tabular}{cc}
	\includegraphics[trim=0mm 0mm 0mm 0mm, width=.3\linewidth]{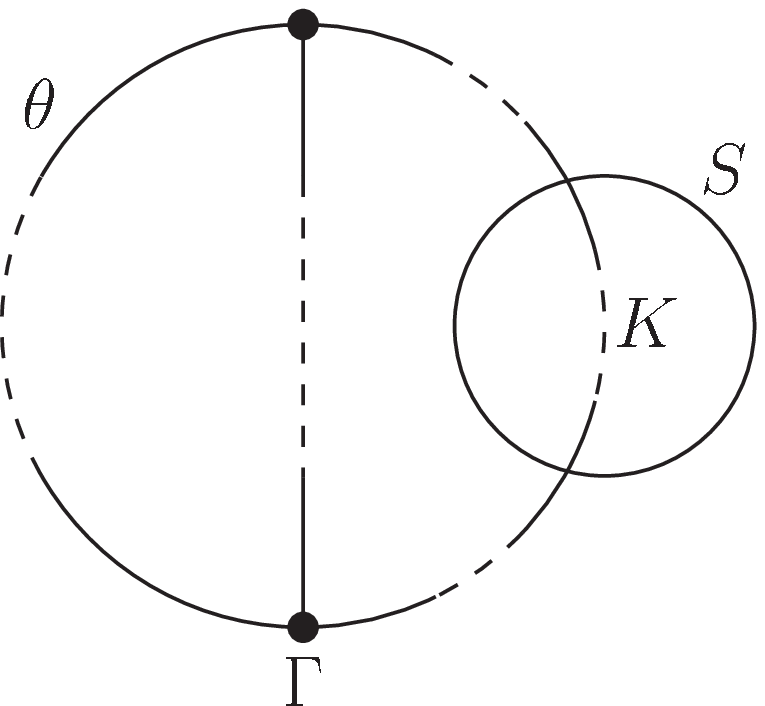}&
	\includegraphics[trim=0mm 0mm 0mm 0mm, width=.5\linewidth]{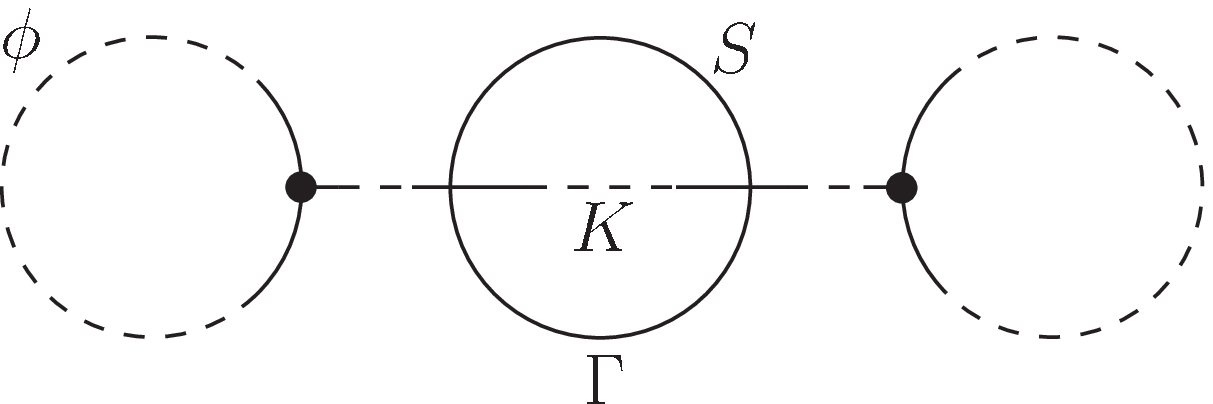}\\
	(1) & (2)\\
	 & \\
	\includegraphics[trim=0mm 0mm 0mm 0mm, width=.45\linewidth]{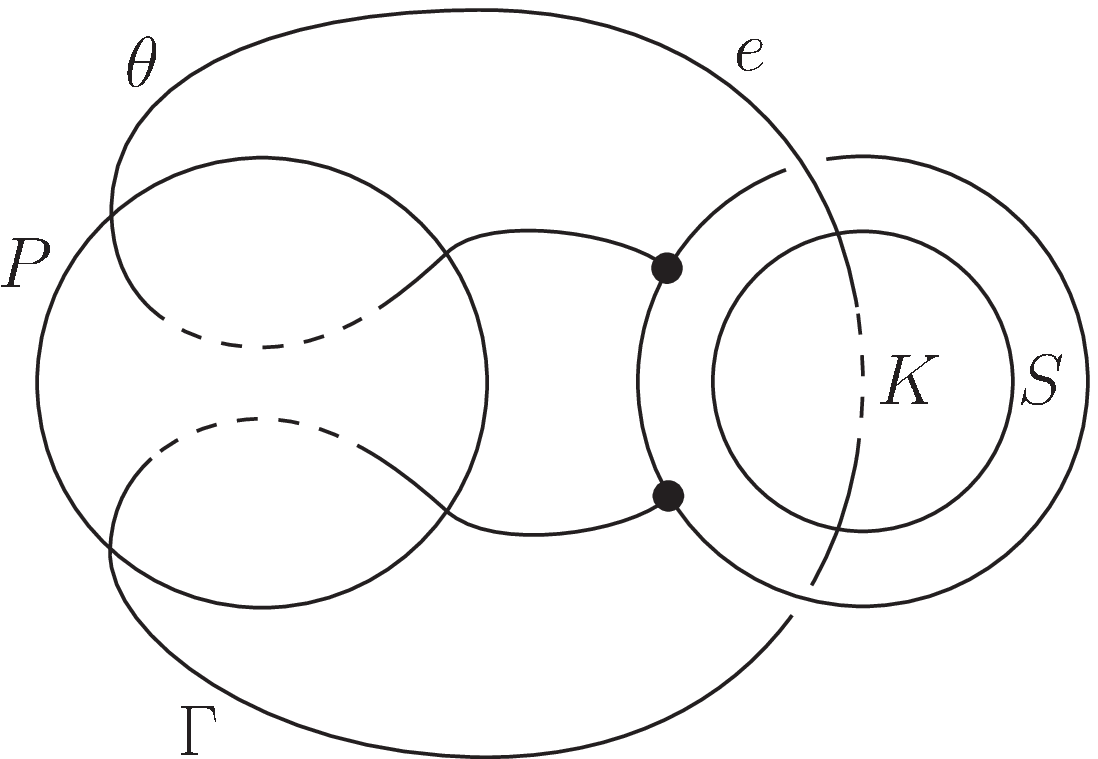}&
	\includegraphics[trim=0mm 0mm 0mm 0mm, width=.45\linewidth]{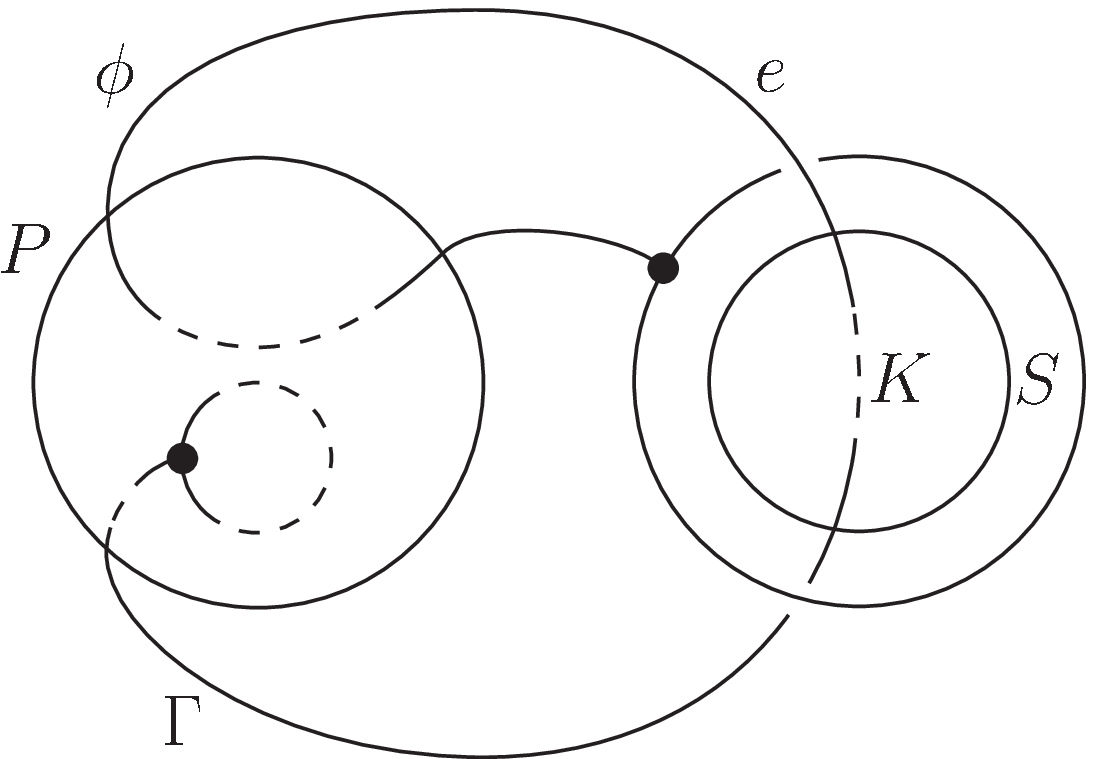}\\
	(3) & (4)\\
	\end{tabular}
	\end{center}
	\caption{Four types of a spine $\Gamma$}
	\label{conf}
\end{figure}

\section{Main proofs}

In this section we prove the following:

\begin{proposition}\label{mainpro} Let $V$ be an irreducible genus two handlebody-knot with has an unknotting tunnel $\gamma$. Suppose that $V$ admits 2-decomposing spheres. Then there is a 2-decomposing sphere $S$ for $V$ which is disjoint from $\gamma$ or intersects it in one point.
\end{proposition}

Let $V$ be an irreducible tunnel number one genus two handlebody-knot, and let $E(V)=S^3-\text{int}N(V)$ be its exterior.
Let $\gamma$ be an unknotting tunnel for $V$. So $W=S^3-\text{int}N(V\cup \gamma)$ is a genus 3 handlebody. Let $\alpha$ be the cocore of the tunnel $\gamma$, that is, a curve on $\partial N(\gamma)$ that bounds a disk in $N(\gamma)$ which intersects $\gamma$ transversely in one point. So $E(V)$ is obtained by adding a 2-handle to $W$ along $\alpha$. Note that the curve $\alpha$ is non-separating in $\partial W$. The arc $\gamma$ may have been slided over itself; in that case it can be expressed as $\gamma=\gamma_1\cup \gamma_2$, where $\gamma_1$ is a simple closed curve, and $\gamma_2$ is an arc with endpoints in $\partial V$ and $\gamma_1$, in this case let $v=\gamma_1\cap \gamma_2$.

\begin{figure}[htbp]
\begin{center}
\includegraphics[trim=0mm 0mm 0mm 0mm, width=.95\linewidth]{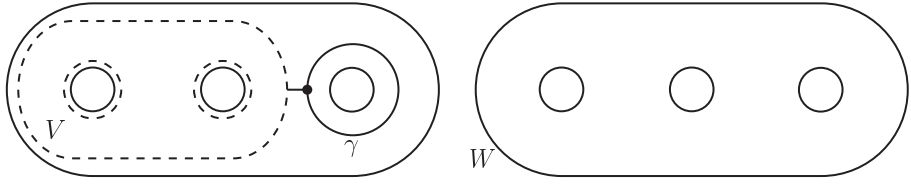}
\end{center}
\caption{$S^3=N(V\cup \gamma) \cup W$}
\label{tunnel}
\end{figure}

Let $S$ be a $2$-decomposing sphere for $V$. So $S\cap V$ consists of $2$ essential disks in $V$, and $\widehat S=S\cap E(V)$ is a properly embedded surface in $E(V)$. Note that because $V$ is a genus 2 handlebody, $S$ bounds a 3-ball $B$ in $S^3$, so that $B\cap \Gamma$ consists of a knotted spanning arc in $B$, where $\Gamma$ is a graph such that $N(\Gamma)=V$. It can be assumed that $\widehat S$ and $\gamma$ meet in general position, that is, $\widehat S$ intersects $\gamma$ transversely in a finite number of points, $\widehat S$ is disjoint from $v$ and $\gamma\cap \partial V$, and $\widehat S\cap N(\gamma)$ consists of a collection of disjoint disks. Then $\widehat S$ meets $\gamma_1$ in, say, $n$ points, and $\gamma_2$ in $m$ points. Define the complexity of $\widehat S$ to be $c(\widehat S)=n+m$.  

Label with $\alpha_1,\alpha_2,\dots,\alpha_n$ the disks of intersection of $\widehat S$ and $N(\gamma_1)$, labeled in order as they occur in $\gamma_1$, starting at $v$ with an arbitrary choice of direction, and label with $\beta_1,\dots,\beta_m$ the disks of intersection between $\widehat S$ and $N(\gamma_2)$, labeled as they occur in $\gamma_2$, going from $v$ to $\partial V$.
Denote the components of $\partial \widehat S$ by $s_1,\ s_2$.

\begin{figure}[htbp]
\begin{center}
\includegraphics[trim=0mm 0mm 0mm 0mm, width=.4\linewidth]{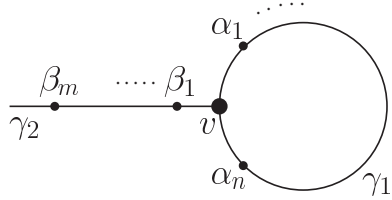}
\end{center}
\caption{Labeling on $\widehat S\cap \gamma_1\cup\gamma_2$}
\label{tunne2}
\end{figure}

Let $\widetilde S$ be the surface $\widehat S-\text{int} N(\gamma)$.
Assume that $\widehat S$ has been isotoped and $\gamma$ has been isotoped and slided, 
to make $c(\widehat S)$ minimal. Suppose also that $S$ is a 2-decomposing sphere so that $\widehat S$ has minimal complexity among all the 2-decomposing spheres for $V$. Assume that $c(\widehat S)\not= 0$, for otherwise we are finished. 

\begin{claim} $\widetilde S$ is incompressible in $W$.
\end{claim}

\begin{proof}
If $\widetilde S$ is compressible, then there is a disk $D$ in $W$ with $D\cap \widetilde S=\partial D$, which is an essential curve on $\widetilde S$. But in $\widehat S$ this curve has to be inessential, i.e. it bounds a disk $D'\subset \widehat S$ which intersects $\gamma$. As $E(V)$ is irreducible, the sphere $D\cup D'$ bounds a 3-ball in $E(V)$, and by interchanging $D$ and $D'$, we obtain a surface $\widehat S'$ isotopic to $\widehat S$ but with lower complexity. 
\end{proof}

Let $D$ be a compression disk for $\partial W$, which exists for $W$ is a handlebody. Consider the intersections between $D$ and $\widetilde S$, which we assume consist of arcs and circles. Simple closed curves of intersection can be removed, for $\widetilde S$ is incompressible. There must exist arcs of intersection, for any compression disk for $\partial W$ meets $\partial N(\gamma)$, for otherwise $\partial V$ would be compressible, and then as $c(\widehat S)\not= 0$, $D$ and $\widetilde S$ intersect. Assume that $D$ has been isotoped to make this intersection minimal. Label the endpoints of the arcs of intersection in $D$ with the labels of the disks of $\widehat S\cap N(\gamma)$ in which the points lie.

Let $\delta$ be an outermost arc of intersection in $D$, which cuts off a disk $D'\subset D$, so that $\partial D'=\delta \cup \eta$, where $\eta$ is an arc on $\partial W$, and the interior of $D'$ is disjoint from $\widetilde S$. 

\begin{claim} \label{tunnelintersection} Both endpoints of $\delta$ are in $\alpha_1$, $\alpha_n$ or $\beta_m$. Furthermore,

\begin{enumerate}
\item If both endpoints are in $\alpha_1$ or $\alpha_n$, then $m=0$, $\eta$ intersects $\partial E(V)$, $\eta\cap \partial E(V)$ is an essential arc in $\partial E(V)-N(\gamma)$, and $\delta$ is an essential arc in $\widetilde S$. 

\item If both endpoints are in $\beta_m$ then $\eta$ intersects $\partial E(V)$, and $\eta\cap \partial E(V)$ is an essential arc in $\partial E(V)-N(\gamma)$, and $\delta$ is an essential arc in $\widetilde S$.

\end{enumerate}
\end{claim}

\begin{proof}
There are several possible cases for $D'$.

\begin{description}
\item [Case 1] An end of $\delta$ is in $s_i$ and the other in $s_j$, $i\not= j$.
\item [Case 2] Both ends of $\delta$ are in $s_i$.
\item [Case 3] An end of $\delta$ is in $s_i$ and the other in $\alpha_1$, $\alpha_n$ or $\beta_m$.
\item [Case 4] An end of $\delta$ is in $\alpha_i$ and the other in $\alpha_{i+1}$ (or $\beta_i$ and $\beta_{i+1}$), and $\eta$ is disjoint from $N(v)$ and from $\partial N(V)$.
\item [Case 5] An end of $\delta$ is in $\beta_1$ and the other in $\alpha_1$ or $\alpha_n$.
\item [Case 6] An end of $\delta$ is in $\alpha_1$ and the other in $\alpha_n$, and $\eta$ intersects $N(v)$ or $\partial E(V)$.
\end{description}

The proof that all these cases cannot happen is identical to the one given in \cite{Eu}, Prop. 2.3, Cases 1-6. There are three more cases.

\begin{description}
\item [Case 7] Both ends of $\delta $ are in $\beta_1$ and $\eta$ is disjoint from $\partial E(V)$.
\end{description}

If $\eta$ is disjoint from $N(v)$ then the intersection between $D$ and $\widetilde S$ would not be minimal, so it must intersect it. The arc $\delta$ cuts off a disk $E$ from $\widehat S$, so that $\partial E=\delta\cup\beta'$, where $\beta'$ is a subarc of $\beta_1$. Note that $E$ may contain disks of intersection with $V$. We may assume that $E$ contains at most one disk of intersection with $V$. So $D'\cup_{\delta}E$ is a disk with boundary $\eta\cup \beta'$. If $\eta$ is disjoint from the curve $\alpha$, then there is a disk $E'\subset N(\gamma)$, with $\partial E'=\eta\cup \beta'$, and which intersects $\gamma_1$ in one point. If $E$ is disjoint from $N(\gamma_1)$, then $D'\cup E\cup E'$ is a sphere in $S^3$ intersecting $\gamma_1$ transversely in one point, which is not possible, so $E$ must intersect $\gamma_1$ an odd number of times. It follows that the disk $E$ must be disjoint from $V$ for otherwise the sphere $D'\cup E\cup E'$ intersects $V$ in one disk, which is not possible. By isotoping $\widehat S$ through the 3-ball bounded by $D'\cup E\cup E'$, we get a sphere isotopic to $\widehat S$, intersecting $\gamma_1$ in a new point, but where the disk $\beta_1$ and the intersections of $E$ with $\gamma_1$ have been eliminated, getting then a sphere with lower complexity.

Then the arc $\eta$ intersects $\alpha$, and  we must have that $n=0$. If the algebraic intersection number of $\eta$ with $\alpha$ is not $\pm 1$, then $D'\cup E$ can be isotoped so that its boundary lies on $\partial N(\gamma_1)$, and then $N(D'\cup E)\cup N(\gamma_1)$ is a punctured lens space, which is not possible. So the arc $\eta$ must intersect $\alpha$ exactly once. In this case the disk $D'$ can be used to isotope $\gamma$ through $\widehat S$, eliminating $\beta_1$ and then reducing $c(\widehat S)$.

\begin{description}
\item [Case 8] Both ends of $\delta $ are in $\alpha_1$ or $\alpha_n$.
\end{description}

If $\eta$ is disjoint from $\partial E(V)$, then the intersection between $D$ and $\widetilde S$ would not be minimal, or we can find a sphere intersecting $\gamma_1$ in one point, or a sphere isotopic to $\widehat S$ but with lower complexity, as in the proof of Case 7. So assume that $\eta$ meets $\partial E(V)$. So we must have $m=0$, and $\gamma_1$ can be slided so that $\gamma$ is just an arc.  The arc $\eta\cap \partial E(V)$ must be essential in $\partial E(V)-N(\gamma)$, for otherwise $\eta$ could be isotoped to lie in $N(\gamma)$. Finally, note that the arc $\delta$ is essential in $\widetilde S$, for otherwise it cuts off a disk $E$ from $\widetilde S$, so that $D'\cup E$ is a disk which can be isotoped to be a compression disk for $E(V)$.

\begin{description}
\item [Case 9] Both ends of $\delta $ are in $\beta_m$.
\end{description}

If $\eta$ is disjoint from $\partial E(V)$, then the intersection between $D$ and $\widetilde S$ would not be minimal or $m=1$ and we are in Case 7. So assume that $\eta$ meets $\partial E(V)$. An argument as in Case 8 shows that $\eta\cap \partial E(V)$ is an essential arc in $\partial E(V)-N(\gamma)$, and that $\eta$ is essential in $\widetilde S$.

This completes the proof of the claim. 
\end{proof}

\begin{claim} \label{tunnelcircle} There is a circle $w$ of $\Gamma$, associated to $\delta$,  which is parallel to an essential circle on $\widehat S$, that is, there is an annulus $A$ in $E(V)$ with interior disjoint from $\widehat S$,  so that one component of $\partial A$ is a curve on $\widehat S$ and the other component is a circle $w$ of $\Gamma$. Furthermore the circle $w$ and the annulus $A$ lie outside of the 3-ball $B$ bounded by $S$. In particular if there is an outermost arc with endpoints in $\alpha_1$ and another one with endpoints in $\alpha_n$, then $n$ must be an even number.
\end{claim}

\begin{figure}[htbp]
\begin{center}
\includegraphics[trim=0mm 0mm 0mm 0mm, width=.4\linewidth]{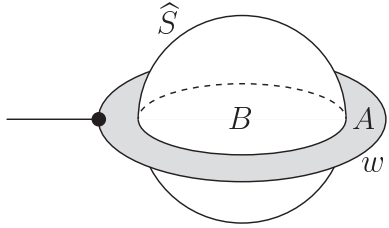}
\end{center}
\caption{Configuration of $A$ and $w$}
\label{annulus}
\end{figure}

\begin{proof}
By Claim \ref{tunnelintersection} we can assume that the endpoints of $\delta$ both lie in $\alpha_1$, $\alpha_n$ or $\beta_m$. If they lie in $\alpha_1$ or $\alpha_n$, then as the arc $\eta$ must intersect $\partial V$, we can assume that $\gamma$ has been slided so that it is just an arc. In all three cases there is a disk $E\subset \partial N(\gamma)$ so that $D'\cup E$ is an annulus with one boundary component $c_1$ in $\widehat S$ and the other $c_0$ in $\partial E(V)$. The curves $c_0$ and $c_1$ are essential in $\widehat S$ and $\partial E(V)$ respectively, by Claim \ref{tunnelintersection}. If $c_0$ is a meridian of $V$, then we can cut off $S$ with $A$ to get two new decomposing spheres. At least one of them must be essential, and the intersection with $\alpha_1$, $\alpha_n$ or $\beta_m$ is eliminated, so we get a 2-decomposing sphere with lower complexity, contradicting the hypothesis.

Suppose then that $c_0$ is not a meridian of $V$. The curve $c_0$ then lies in the boundary of a solid torus $V'\subset V$, this because $S$ divides $V$ into solid tori and 3-balls. If $c_0$ wraps twice or more around this solid torus, we have a punctured lens space, which is formed by the union of $V'$, the annulus $A$, plus a disk in $S$ bounded by $c_1$. As this is not possible, $c_0$ wraps just once around $V'$, and then it is isotopic to its core $w$. This implies that $w$ is isotopic to the curve $c_1$. Finally note that the annulus $A$ can not be contained in $B$, for $B$ contains just an arc of $\Gamma$. \end{proof}

\begin{claim} \label{tunnelparallelarcs} There can not be 3 parallel arcs in $D$, where one of the arcs is outermost in $D$, and the endpoints of the arcs are all in the $\alpha_i's$ or all in the $\beta_i's$.
\end{claim}

\begin{proof}
Assume that there are 3 parallel arcs in $D$, say $\delta_1$, $\delta_2$, $\delta_3$, where $\delta_1$ is an outermost arc with both endpoints in $\alpha_1$, and $\delta_2$, and $\delta_3$ have both endpoints in $\alpha_2$ and $\alpha_3$, respectively. If the endpoints of $\delta_1$ are in $\alpha_n$ or $\beta_m$, the proof is similar. Let $D_1$ be the disk cut off by $\delta_1$ in $D$, and let $D_2$ ($D_3$) be the disks in $D$ determined by $\delta_1$ and $\delta_2$ ($\delta_2$ and $\delta_3$ respectively). As in Claim \ref{tunnelparallelarcs}, there is a disk $E_1\subset \partial N(\gamma)$, so that $A_1=D_1\cup E_1$ is an annulus in $E(V)$ with interior disjoint from $\widehat S$,  so that one component of $\partial A$ is a curve $c_1$ on $\widehat S$ and the other component is a circle of $\Gamma$. There are disks $E_2$, $E_3$ in $\partial N(\gamma)$, so that $A_2=D_2\cup E_2$, $A_3=D_3\cup E_3$ are two annuli. The annulus $A_2$ is contained in the 3-ball $B$ bounded by $S$, and the annulus $A_3$ is in the complement of this ball, $\partial A_2=c_1\cup c_2$ and $\partial A_3=c_2\cup c_3$, where $c_2$ and $c_3$ are essential simple closed curves in $\widehat S$. The curves $c_2$ and $c_3$ bound an annulus $F\subset \widehat S$. Consider the sphere $\Sigma = (S-F)\cup A_3$. This is a 2-decomposing sphere for $V$; this sphere is not trivial, because the knotted arc of $\Gamma$ lying in $B$ still remain as a knotted arc inside the 3-ball bounded by $\Sigma$. If the annuli $F$ and $A_3$ are parallel, then in fact $\widehat S$ and $\widehat \Sigma$ are isotopic, but $c(\widehat \Sigma) < c(\widehat S)$. If these annuli are not parallel, then $\Sigma$ is a new essential 2-decomposing sphere with $c(\widehat \Sigma) < c(\widehat S)$. 
\end{proof}

Note that in the above proof, the curves $c_1$ and $c_2$ bound an annulus $F_1$ in $\widehat S$. But in this case the new sphere $(S-F_1)\cup A_2$ may not be essential.

\begin{claim} \label{bounds} We have that $n\leq 4$, or $m\leq 2$.
\end{claim}

\begin{proof}
In this and next claim we do an outermost fork argument, as in \cite{Sch}. $D\cap \widetilde S$ consists of a collection of arcs in $D$. We construct a tree in $D$ as follows: take a vertex for each region of $D-\widetilde S$, and connect two vertices if their respective regions are adjacent, that is, they have an arc of $D\cap \widetilde S$ in common. The resultant graph $G$ is a tree, because $D$ is a disc. The ends of the tree, that is the vertices of degree one, correspond to the outermost regions of $D$.

A branch of $G$ is a trajectory that begins in a vertex of degree one of $G$ and finishes in a vertex of degree $>2$, so that the intermediate vertices of the branch are all of degree 2. If all the vertices of $G$ are of degree 1 or 2, that is, $G$ is a trajectory, then all the arcs are parallel. 
By Claim \ref{tunnelparallelarcs} there can not be 3 parallel arcs, so $G$ has at most 3 vertices and 2 edges, which implies that $n\leq 2$ or $m\leq 2$.

If $G$ is not a trajectory, let $G'$ be the graph obtained by eliminating the branches, that is, by clearing the vertices of degree 1 and 2 of the branches together with the corresponding edges.
Let $x$ be a vertex of degree 1 of $G'$ (if vertices of degree 1 do not exist,  let $x$ the unique vertex of $G'$). Then at least two branches arrive at $x$, and say, let  $r_1$ and $r_2$ be two adjacent branches arriving at $x$, and let $\epsilon$ an arc of $\partial D$ that goes from the outermost region of $r_1$ to the one of $r_2$. 

This shows that there are two adjacent sets of parallel arcs of intersection in $D$, each containing an outermost arc. If the corresponding outermost arcs both have endpoints labeled $\alpha_1-\alpha_1$, 
then the arc $\epsilon$ must cross labels $1,2,\dots,n-1,,n,n,n-1,\dots,2,1$, and perhaps more labels between $n$ and $n$.  Any arc of intersection that leaves these labels corresponds to an edge of $r_1$ or $r_2$, by the selection of the branches. This implies that $r_1\cup r_2$ has at least $2n$ edges, and then at least one of the branches has $n$ or more edges, that is, correspond to $n$ parallel arcs.
Then there will be 3 parallel edges, which is not possible by Claim \ref{tunnelparallelarcs}, so we must have that $n\leq 2$. Similarly, if the  outermost arcs both have endpoints labeled $\alpha_n-\alpha_n$, or $\beta_m-\beta_m$, then there will be 3 parallel edges, unless $n\leq 2$ or $m \leq 2$. Remember that if an outermost arc have endpoints in $\alpha_1$ or $\alpha_n$, then $m=0$. If the corresponding outermost arcs have endpoints $\alpha_1-\alpha_1$ and $\alpha_n-\alpha_n$, Then there will be 3 parallel edges, unless $n\leq 4$. 
\end{proof}

\begin{claim} \label{morebounds} If $m\not= 0$ then $n=0$.
\end{claim}

\begin{proof}
Suppose $n\not= 0$. Any outermost arc must have both endpoints in $\beta_m$. 
Let $G$ and $G'$ be the graphs constructed in Claim \ref{bounds}. If $G$ is a trajectory then clearly $n=0$; so suppose it is not a trajectory and then $G'$ is non-empty. Let $x$ be a vertex of degree 1 of $G'$ (if vertices of degree 1 do not exist,  let $x$ the unique vertex of $G'$). Then at least two branches arrive at $x$, say $r_1$ and $r_2$ are two adjacent branches arriving at $x$, and let $\epsilon$ be an arc of $\partial D$ that goes from the outermost region of $r_1$ to the one of $r_2$. $\epsilon$ must cross labels $m,m-1,\dots,2,1,1,2,\dots,n-1,n,1,2,\dots,n-1,n,1,2,\dots,m$, where the sequence $1,2,\dots,n-1,n$ can be repeated several times. With a simple orientation argument, it can be shown that for a label $i\in \{ 1,2,\dots,n-1,n\}$, there can not be an arc with ends labeled $i-i$. If in one of the branches arriving at the vertex $x$, there are at least $m +1 +n/2$ parallel arcs, there there will be two parallel arcs with ends labeled $n/2$ and $(n+1)/2$. The disk bounded by these parallel arcs forms what is commonly called a Scharlemann cycle (\cite{Sch}). So, there will be a Scharlemann cycle formed by two parallel arcs, unless all the branches arriving at $x$ have exactly $m+n/2$ parallel arcs. In this case the region $F$ of $D$ corresponding to $x$, have arcs with endpoints labeled $n/2, (n+1)/2,n/2, (n+1)/2,\dots,n/2, (n+1)/2$ in this order, that is, it is a Scharlemann cycle. As usual, this implies the existence of a punctured lens space embedded in our ambient manifold. This is formed by taking a regular neighborhood $N(S\cup H \cup F)$, where $H$ is a 1-handle attached to $S$ which consists of  the part of $N(\gamma_1)$ bounded by $\alpha_{n/2}$ and $\alpha_{(n+1)/2}$. As this is impossible, we conclude that $n$ must be $0$. 
\end{proof}

\begin{claim} \label{nsmall} Suppose that $n\not=0$. Then $n\leq 2$.
\end{claim}

\begin{proof}
By Claim \ref{bounds} we can assume that $n\leq 4$. Suppose first that $n=4$.
By Claim \ref{tunnelparallelarcs} there are two pairs of parallel edges in $D$, so that the ends of the edges of one of the pairs are labeled $\alpha_1-\alpha_1$ and $\alpha_2-\alpha_2$, and the ends of the other arcs are labeled  $\alpha_4-\alpha_4$ and $\alpha_3-\alpha_3$. As in the proof of Claim \ref{tunnelparallelarcs}, there are two annuli in $B$, determined by the pairs of parallel edges. One such annulus $A_2$ has as boundary the curves $c_1$ and $c_2$, which are essential in $\widehat S$. The other annulus $A_3$ has as boundary curves $c_3$, $c_4$, which are also essential in $\widehat S$. The curves $c_1$ and $c_2$ bound an annulus $F$ in $\widehat S$. By interchanging $F$ and $A_2$ we get a new 2-decomposing sphere, which will be essential unless $A_2$ is an annulus that follow the knotted arc of $\Gamma$ lying in $B$. Something similar can be said about the annulus $A_3$, where $c_3$ and $c_4$ bound an annulus $F'$ in $\widehat S$. So we get a new 2-decomposing sphere with lower complexity, unless $A_2$ and $A_3$ are parallel and follow the knotted arc lying in $B$. If this happens, then suppose, say, that $F' \subset F$. Then the torus $F\cup A_3$ can be isotoped to be disjoint from $\gamma$. But it is an essential torus in the complement of $V\cup N(\gamma)$, which is impossible for $W$ is a handlebody. 

Suppose now that $n=3$. In this case either there are 3 parallel edges, contradicting Claim \ref{tunnelcircle}, or there is an outermost arc with endpoints in $\alpha_1$, and another outermost arc with endpoints in $\alpha_3$. But this contradicts Claim \ref{tunnelcircle}. 
\end{proof}

\begin{claim} \label{msmall} If $m\not= 0$ then $m=1$.
\end{claim}

\begin{proof}
By Claim \ref{bounds} suppose that $m=2$. Then there is a pair of parallel edges in $D$, one with labels $\beta_2-\beta_2$, and the other with labels $\beta_1-\beta_1$. As in the proof of Claim \ref{nsmall}, there is an annulus in $B$, which implies that there is another 2-decomposing sphere, or that there is an incompressible torus in $W$, which is not possible. 
\end{proof}

\begin{claim} \label{menoscasos} The cases $n=2$, $m=0$, or $m=1$, $n=0$, are not possible.
\end{claim}

\begin{proof}
Suppose first that $n=2$ and then $m=0$. If there is a pair of parallel edges in $D$, one with labels $\alpha_1-\alpha_1$, and the other with labels $\alpha_2-\alpha_2$, then proceed as in Claim \ref{msmall}, to show that this is not possible. Suppose then that there is an outermost arc with labels  $\alpha_1-\alpha_1$ and another one with labels  $\alpha_2-\alpha_2$. Then by Claim \ref{tunnelcircle} there are two annuli $A_1$, $A_2$ outside the 3-ball $B$, so that $\partial A_i=c_i\cup w_i$, where $c_i $ is an essential curve embedded in  $\widehat S$, and $w_i$ is a curve on $\partial V$, for $i=1,2$. Note that $c_1\not= c_2$, but that $w_1$ and $w_2$ could be isotopic curves in $\partial V$, in fact, they will be isotopic curves if and only if $\Gamma$ is a theta-curve. Note that the subarc of $\gamma$ going from $\partial V$ to $\alpha_1$ ($\alpha_2$) is contained in the annulus $A_1$ ($A_2$). The curves $c_1$, $c_2$ divide $\widehat S$ into 3 annuli. Let $C_1$ ($C_2$) be the annulus whose boundary consists of the curve $c_1$ ($c_2$) and a component of $\partial \widehat S$ (the other component of $\partial \widehat S$) and let $C_3$ be the annulus bounded by $c_1$ and $c_2$.

Suppose first that $\Gamma$ is a theta-curve. Let $C_1'=C_1\cup A_1$ and $C_2'=C_2\cup A_2$, and push them to be disjoint from the tunnel $\gamma$ and from $B$. Note that $C_1'$ and $C_2'$ are disjoint annuli properly embedded in the handlebody $W$, so they are $\partial$-compressible. Note that there is a $\partial$-compression disk for one of the annuli which is disjoint from the other one. Say, there is a disk $E$ contained in $W$, so that $\partial E=\nu\cup \mu$, where $\nu$ is an arc on $\partial W$, $\mu$ is a spanning arc in $C_1'$, and $E$ is disjoint from $C_2'$. Note that $C_1' \cup C_2'$ divides $W$ into two handlebodies $W_1$ and $W_2$, where, say, $B\cap W$ is contained in $W_1$. Note that $E$ must be contained in $W_2$, for otherwise the arc $\nu$  would intersect both components of $\widehat S$, i.e. it would intersect $C_2'$. By using $E$, we may assume that there is an arc of $\Gamma \cap W_2$ that is isotopic to the arc $\mu$ on $E$, and by taking a neighborhood of $C_1'$ that contains $E$, it is not difficult to see that there is a disk $F$ in $W_2$, whose boundary is an essential curve on $V$,  i.e., $V$ would be reducible, which is not possible.

Suppose now that $\Gamma$ is a handcuff graph. Let $C_1'$ and $C_2'$ be defined as above. Consider the annulus $C_3'=A_1 \cup C_3 \cup A_2$, and push it  to be disjoint from the tunnel $\gamma$, from $B$ and from $C_1'\cup C_2'$. Note that $C_1' \cup C_2' \cup C_3'$ divides $W$ into two handlebodies $W_1$ and $W_2$, where, say, $B\cap W$ is contained in $W_1$. There is a compression disk $E$ for one of the annuli which is disjoint form the other two annuli. If $E$ is contained in $W_1$, it will be a $\partial$-compression disk for $C_3'$ but it would imply that the tunnel $\gamma$ is isotopic to an arc on $C_3'$, i.e., it would be disjoint from $\widehat S$. If $E$ is contained in $W_2$, then it would be a $\partial$-compression disk for $C_1'$ or $C_2'$, and as before, this will imply that $V$ is reducible.

Suppose now that $m=1$ and $n=0$.  There is an outermost arc in $D$ with labels  $\beta_1-\beta_1$. By Claim \ref{tunnelcircle} there is one annulus $A_1$ outside the 3-ball $B$, so that $\partial A_1=c_1\cup w_1$, where $c_1 $ is an essential curve embedded in  $\widehat S$, and $w_1$ is a curve on $\partial V$. Note that the subarc of $\gamma$ going from $\partial V$ to $\alpha_1$ is contained in the annulus $A_1$. The curve $c_1$ divides $\widehat S$ into 2 annuli $C_1$  and $C_2$. Let $C_1'=C_1\cup A_1$ and $C_2'=C_2\cup A_1$, and push them to be disjoint from the tunnel $\gamma$ and from $B$. Note that $C_1'$ and $C_2'$ are disjoint annuli properly embedded in the handlebody $W$, so they are $\partial$-compressible. An argument as in the previous cases shows that this is not possible.
\end{proof}

So we conclude that if $S$ is not disjoint from $\gamma$, then it intersects it once, and $\gamma$ is just one arc. This completes the proof of Proposition \ref{mainpro}.

\section{Conclusion}

In this section, we characterize composite tunnel number one genus two handlebody-knots.
Recall that $V$ is a tunnel number one genus two handlebody-knot in $S^3$ whose exterior is boundary-irreducible, $\gamma$ is an unknotting tunnel for $V$, and $S$ is a 2-decomposing sphere for $V$ which intersects $\gamma$ in at most one point by the previous section.
$S$ bounds a 3-ball $B$ such that a spine $\Gamma$ of $V$ intersects $B$ in a knotted arc $k$.

There are four cases to consider:

\begin{enumerate}
\item $S\cap \gamma=\emptyset$, and $\Gamma-k$ is connected.
\item $S\cap \gamma=\emptyset$, and $\Gamma-k$ is not connected.
\item $|S\cap \gamma|=1$, and $\Gamma-k$ is connected.
\item $|S\cap \gamma|=1$, and $\Gamma-k$ is not connected.
\end{enumerate}

Note that in Cases (1) and (2), the tunnel $\gamma$ lies on the 3-ball $B$. $N(\Gamma\cup \gamma)$ is a genus 3 handlebody and let $W=S^3-\text{int} N(\Gamma\cup \gamma)$ be the complementary genus 3 handlebody.

In Case (1), we take a spine $\Gamma$ of $V$ as a theta-curve graph.
$\Gamma$ is decomposed into a theta-curve graph $\theta$ and a knot $K$. $N(\Gamma\cup \gamma)$ is decomposed by two disks of $S\cap N(\Gamma\cup \gamma)$ into two solid tori $V_1$ and $V_2$.
The handlebody $W$ is decomposed by the separating annulus $\widehat S =S\cap W$ into two genus 2 handlebodies $W_1$ and $W_2$ (Figure \ref{case1}).

Since $W_1$ is a handlebody, $\theta$ has a tunnel number 0.

Let $E$ be a boundary-compressing disk for the separating annulus $\widehat S$ in $W$.
Note that $E \subset W_2$.
We consider a torus $T$ obtained from $\partial V_2$ by isotoping it into $int V_2$ slightly.
Then $T$ bounds a solid torus $X$ in $V_2$ and $k\cap X$ is an unknotted arc in $X$.
If we attach a pair $(B',t)$ of a 3-ball and an unknotted arc to $(B,k)$, then we have a pair of the 3-sphere and a knot $K$.
Since $W_2$ is a genus 2 handlebody and $E$ cuts $W_2$ into a solid torus, $T$ bounds a solid torus $Y=S^3-\text{int} X$. The arc $t$ is isotopic to an spanning arc of the annulus $\widehat S$, which is in turn, by isotoping it through $E$, isotopic to an arc lying  on $T$. Hence $K$ is a $(1,1)$-knot, for $K\cap X=k\cap X$ is unknotted in $T$ and $K\cap Y$ is unknotted in $Y$.

\begin{figure}[htbp]
\begin{center}
\includegraphics[trim=0mm 0mm 0mm 0mm, width=.5\linewidth]{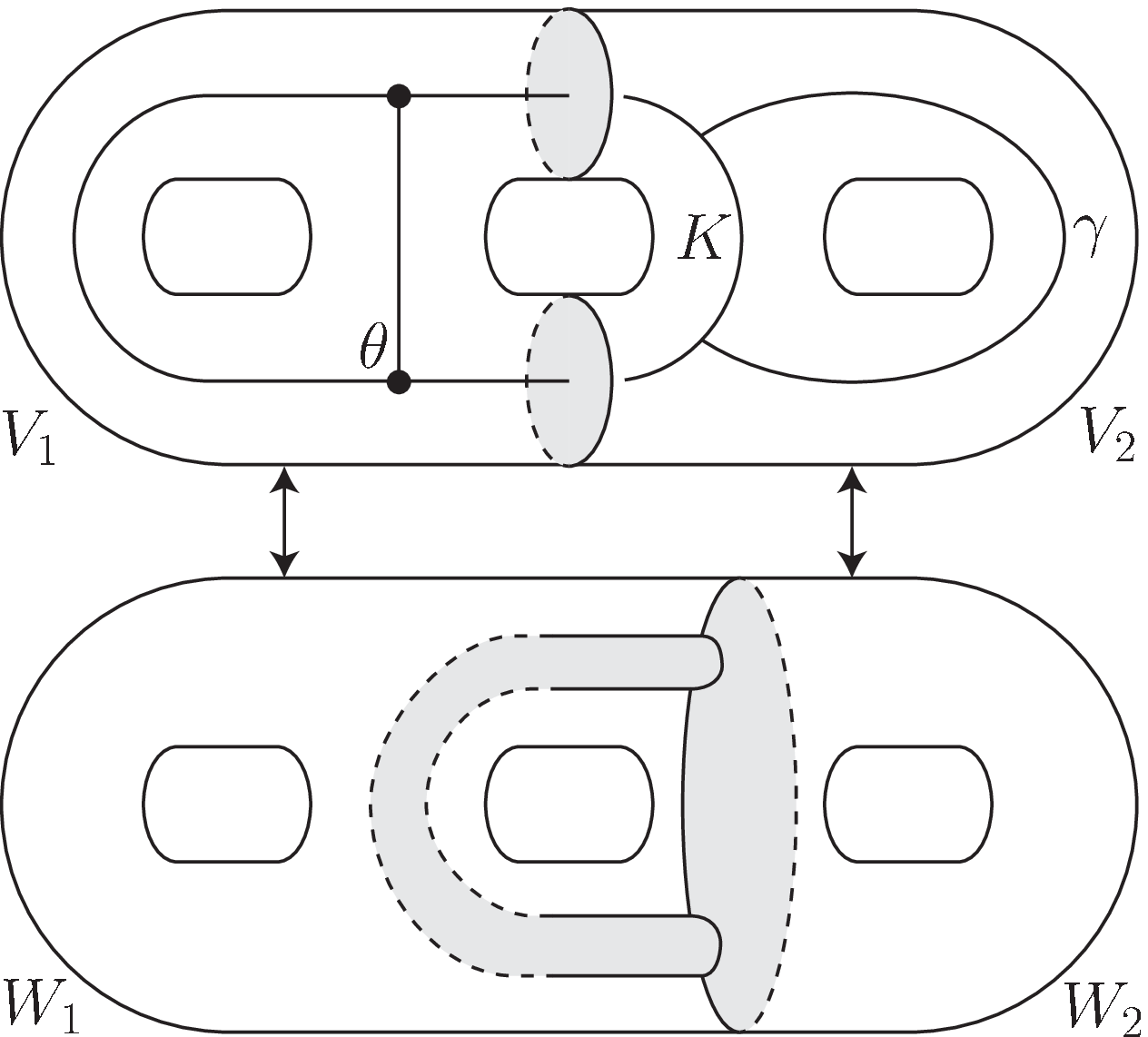}
\end{center}
\caption{Case (1)}
\label{case1}
\end{figure}

In Case (2), we take a spine $\Gamma$ of $V$ as a handcuff graph.
$\Gamma$ is decomposed into a handcuff graph $\phi$ and a knot $k$. $N(\Gamma\cup \gamma)$ is decomposed by two disks of $S\cap N(\Gamma\cup \gamma)$ into three solid tori $V_1$, $V_2$ and $V_3$.
The handlebody $W$ is decomposed by the separating annulus $\widehat S =S\cap W$ into two genus 2 handlebodies $W_1$ and $W_2$ (Figure \ref{case2}).
Similarly to Case 1, $\phi$ has a tunnel number 0 and $k$ is a $(1,1)$-knot.

\begin{figure}[htbp]
\begin{center}
\includegraphics[trim=0mm 0mm 0mm 0mm, width=.5\linewidth]{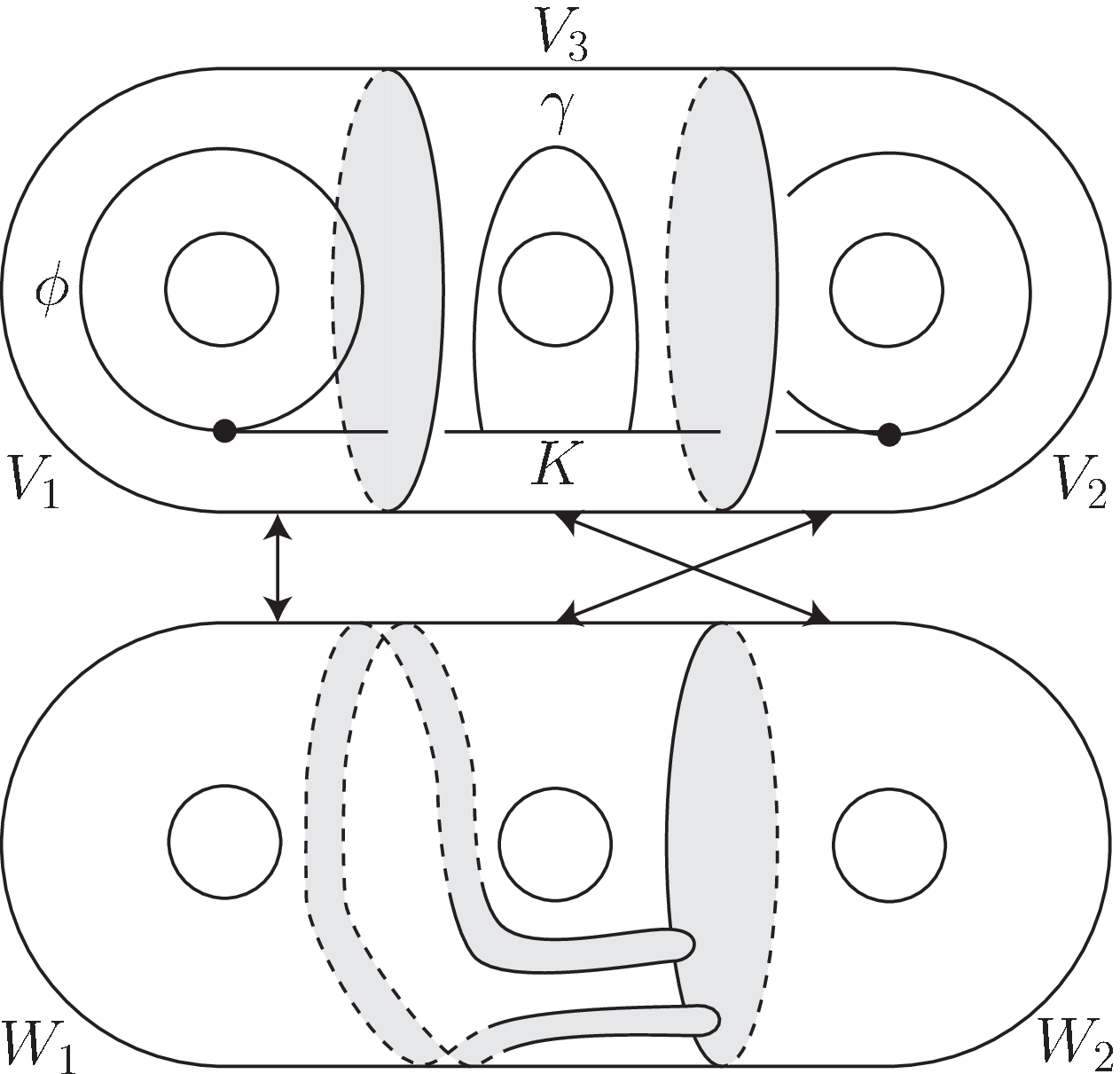}
\end{center}
\caption{Case (2)}
\label{case2}
\end{figure}

In Case (3), we take a spine $\Gamma$ of $V$ as a theta-curve graph.
$\Gamma$ is decomposed into a theta-curve graph $\theta$ and a knot $k$.
The 2-decomposing sphere $S$ for $V$ intersects $N(\Gamma\cup \gamma)$ in 3 disks one of which intersects $\gamma$ in a single point, which divide $V$ into a solid torus  and a 3-ball. Note that $\widetilde S=S\cap W$ is a pair of pants.

Note that by Claim \ref{tunnelintersection} there is a boundary-compressing disk $E$ for $\widetilde S$ in $W$, with interior disjoint from $B$ by Claim  \ref{tunnelcircle}, and so that  $E\cap \widehat S$ is an arc with both endpoints in the disk of intersection of $S$ with $N(\gamma)$. By isotoping $S$ along $E$, $S$ intersects $N(\Gamma\cup \gamma)$ in 2 disks and one annulus $A_1$ which intersects $\gamma$ in a single point, and $S$ intersects $W$ in two annuli $A_2$, $A_3$. $N(\Gamma\cup \gamma)$ is decomposed by two disks of $S\cap N(\Gamma\cup \gamma)$ and the annulus $A_1$ into two solid tori $V_1$ and $V_2$.
It follows from Claim \ref{tunnelcircle} that the core of $A_1$ is parallel to a cycle $w$ of $\theta-k$ in $V_1$. The handlebody $W$ is decomposed by two non-separating annuli $A_2$, $A_3$ into two genus 2 handlebodies $W_1$ and $W_2$ (Figure \ref{case3}).

\begin{figure}[htbp]
\begin{center}
\includegraphics[trim=0mm 0mm 0mm 0mm, width=.5\linewidth]{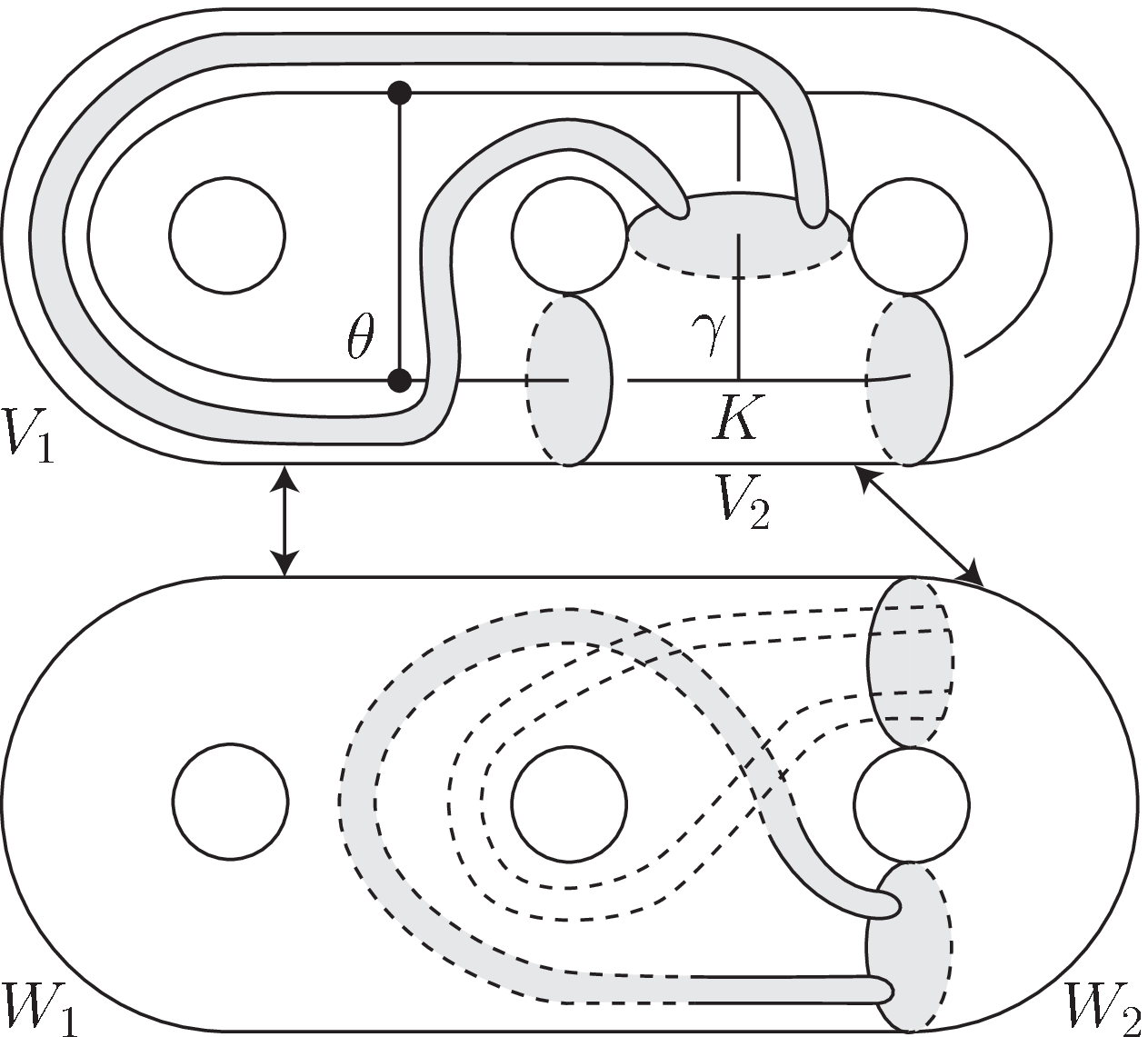}
\end{center}
\caption{Case (3)}
\label{case3}
\end{figure}

Now we show that $k$ is a 2-bridge knot.
We note that $k$ is parallel to an arc on $\partial V_2$, $A_1$ is boundary-compressible, and $A_2$, $A_3$ are also boundary-compressible in $W_2$.
This shows that $k$ has a 2-bridge decomposition as Figure \ref{2-bridge}.

\begin{figure}[htbp]
\begin{center}
\includegraphics[trim=0mm 0mm 0mm 0mm, width=.7\linewidth]{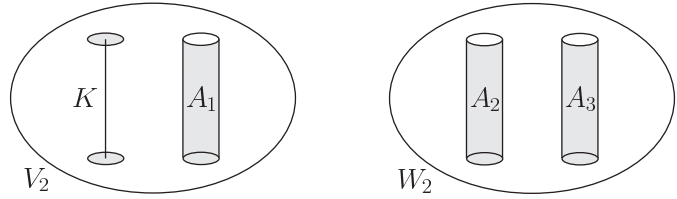}
\end{center}
\caption{A 2-bridge decomposition of $K$}
\label{2-bridge}
\end{figure}

Since the cycle of $\theta-k$ is parallel to the core of $A_1$, the cycle can be put on $S$.
Let $S'$ be a 2-sphere which is obtained from $S$ by putting the cycle on it.
Then we take another 2-sphere $P$ which is parallel to $S'$ and bounds a 2-string tangle.
Since $W_1$ is a genus 2 handlebody, the 2-string tangle is free.
Moreover, if the tangle is not essential, say, it is a rational tangle,
then the handlebody-knot is reducible, this can be seen by untwisting the rational tangle
around the loop $w$.
The complementary tangle is a Hopf tangle with two edges connected sum with a 2-bridge knot.

In Case (4), we take a spine $\Gamma$ of $V$ as a handcuff graph.
$\Gamma$ is decomposed into a theta-curve graph $\phi$ and a knot $k$.
The 2-decomposing sphere $S$ for $V$ intersects the genus 3 handlebody $N(\Gamma\cup \gamma)$ in 3 disks one of which intersects $\gamma$ in a single point, and $\widetilde S=S\cap W$ is a pair of pants.

Note that by Claim \ref{tunnelintersection} there is a boundary-compressing disk $E$ for $\widetilde S$ in $W$, with interior disjoint from $B$ by Claim  \ref{tunnelcircle}, and so that  $E\cap \widehat S$ is an arc with both endpoints in the disk of intersection of $S$ with $N(\gamma)$. By isotoping $S$ along $E$, $S$ intersects $N(\Gamma\cup \gamma)$ in 2 disks and one annulus $A_1$ which intersects $\gamma$ in a single point, and $S$ intersects $W$ in two annuli $A_2$, $A_3$.

The handlebody $N(\Gamma\cup \gamma)$ is decomposed by two disks of $S\cap N(\Gamma\cup \gamma)$ and the annulus $A_1$ into three solid tori $V_1$, $V_2$ and $V_3$. It follows from Claim \ref{tunnelcircle} that the core of $A_1$ is parallel to a cycle of $\phi-k$ in $V_1$. The handlebody $W$ is decomposed by two non-separating annuli $A_2$, $A_3$ into two genus 2 handlebodies $W_1$ and $W_2$ (Figure \ref{case4}).
Similarly to Case 3, $k$ is a 2-bridge knot, and $\phi$ is decomposed by a 2-sphere $P$ into a free tangle and a Hopf tangle with one edge.
If the free tangle is not essential, then there is a compression disk for $P$, which separates
the arc and the other circle of the graph, which then implies that the circle bounds a disk,
i.e., again the handlebody-knot is reducible.

\begin{figure}[htbp]
\begin{center}
\includegraphics[trim=0mm 0mm 0mm 0mm, width=.5\linewidth]{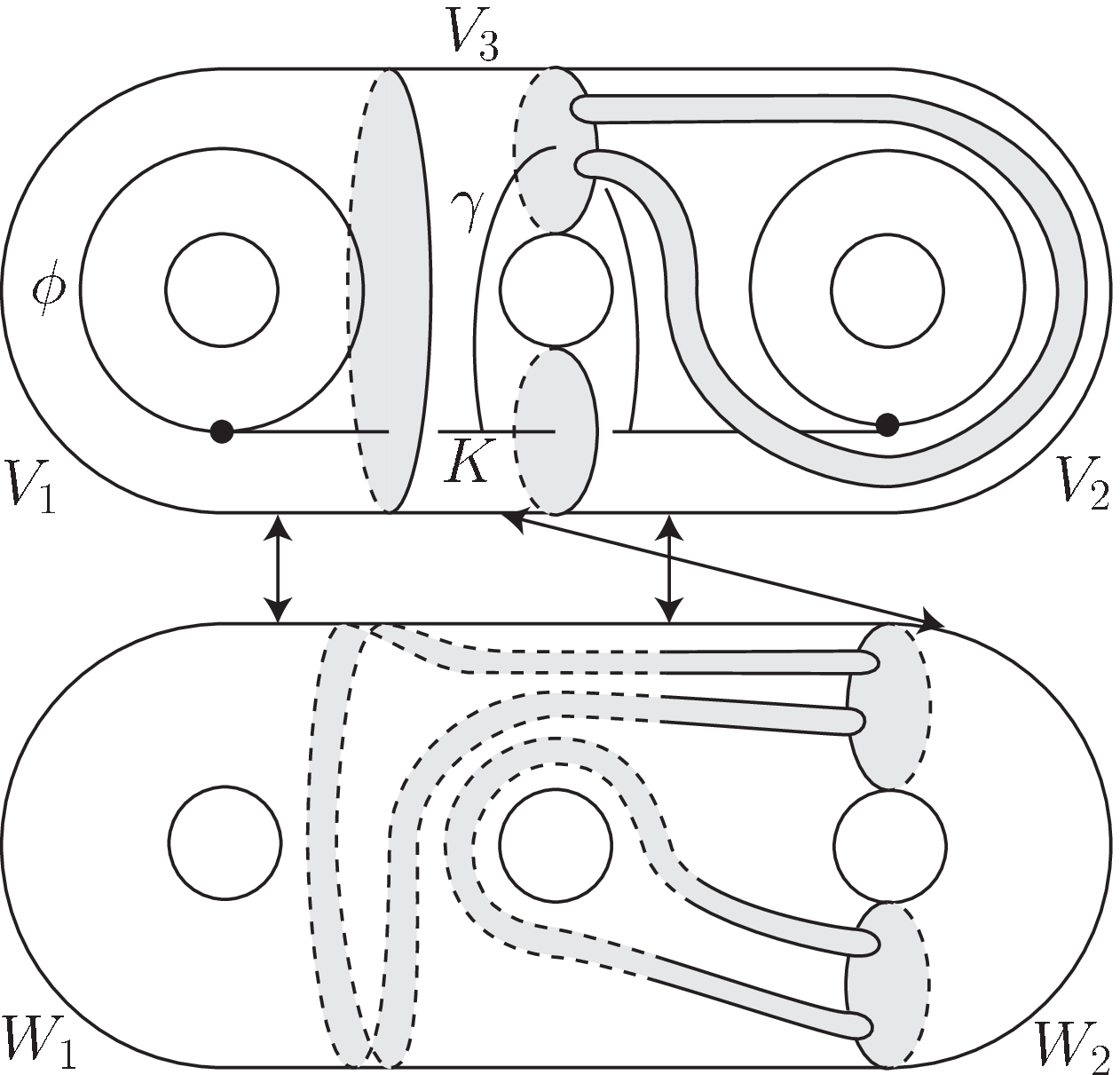}
\end{center}
\caption{Case (4)}
\label{case4}
\end{figure}

\bibliographystyle{amsplain}

\begin{thebibliography}{10}


\bibitem{Eu} M. Eudave-Mu\~noz, {\em On non-simple 3-manifolds and 2-handle addition}, Topology Appl. {\bf 55} (1994), 131--152. 

\bibitem{EU} M. Eudave-Mu\~{n}oz and Y. Uchida, {\em Non-simple links with tunnel number one}, Proc. Amer. Math. Soc. {\bf 124} (1996), 1567--1575.

\bibitem{GOT} H. Goda, M. Ozawa and M. Teragaito, {\em On tangle decompositions of tunnel number one links}, J. Knot Theory and its Ramification {\bf 8} (1999), 299--320.

\bibitem{GR} C. McA. Gordon and A. W. Reid, {\em Tangle decompositions of tunnel number one knots and links}, J. Knot Theory and its Ramification {\bf 4} (1995), 389--409.

\bibitem{H} Y. Hashizume, {\em On the uniqueness of the decomposition of a link}, Osaka Math. J. {\bf 10} (1958), 283--300, Erratum ibid. {\bf 11} (1959), 249.

\bibitem{IKMS} A. Ishii, K. Kishimoto, H. Moriuchi and M. Suzuki, {\em A table of genus two handlebody-knots up to six crossings}, J. Knot Theory Ramifications {\bf 21}, (2012), 1250035.

\bibitem{IKO} A. Ishii, K. Kishimoto and M. Ozawa, {\em Knotted handle decomposing spheres for handlebody-knots}, arXiv:1211.4458.

\bibitem{J} A. Jones, {\em Composite two-generator links have a Hopf link summand}, Topology Appl. {\bf 67} (1995), 165--178.

\bibitem{KO} Y. Koda, M. Ozawa, {\em Essential surfaces of non-negative Euler characteristic in genus two handlebody exteriors}, arXiv:1212.5928.

\bibitem{L} T. Li, {\em Rank and genus of 3-manifolds}, J. Amer. Math. Soc. {\bf 26} (2013), 777--829.

\bibitem{MT} S. Matveev and V. Turaev, {\em A semigroup of theta-curves in 3-manifolds}, Moscow. Math. J. {\bf 11} (2011), 805--814.

\bibitem{M} K. Morimoto, {\em On composite tunnel number one links}, Topology Appl. {\bf 59} (1994), 59--71.

\bibitem{Mo1} T. Motohashi, {\em A prime decomposition theorem for $\theta_n$-curves in $S^3$}, Topology Appl. {\bf 83} (1998), 203--211.

\bibitem{Mo2} T. Motohashi, {\em A prime decomposition theorem for handcuff graphs in $S^3$}, Topology Appl. {\bf 154} (2007), 3135--3139.

\bibitem{N} F. H. Norwood, {\em Every two generator knot is prime}, Proc. Amer. Math. Soc. {\bf 86} (1982), 143--147.

\bibitem{Sch} M. Scharlemann, {\em Tunnel number one knots satisfy the Poenaru conjecture}, Topology Appl. {\bf 18} (1984), 235--258.

\bibitem{S} H. Schubert, {\em Die eindeutige Zerlegbarkeit eines Knoten in Primknoten}, Sitzungsber. Akad. Wiss. Heidelberg, math.-nat. KI. {\bf 3} (1949) Abh: 57--104.

\bibitem{Su1} S. Suzuki, {\em On linear graphs in 3-sphere}, Osaka J. Math. {\bf 7} (1970), 375--396.

\bibitem{Su} S. Suzuki, {\em A prime decomposition theorem for a graph in the 3-sphere}, Topology and Computer Science, (1987), 259--276.

\bibitem{T1} Y. Tsukui, {\em On surfaces in 3-space}, Yokohama Math.  J. {\bf 18} (1970), 93--104.

\bibitem{T} Y. Tsukui, {\em On a prime surface of genus 2 and homeomorphic splitting of 3-spheres}, Yokohama Math.  J. {\bf 23} (1975), 63--75.

\bibitem{W} F. Waldhausen, {\em Heegaard-Zerlegungen der 3-Sph\"{a}re}, Topology {\bf 7} (1968), 195--203.

\end{thebibliography}

\end{document}